\documentclass[10pt,a4paper]{article}
\usepackage[latin1]{inputenc}
\usepackage{listings}
\usepackage{amsmath, amsthm, amssymb}
\usepackage{amssymb}
\usepackage{enumerate}
\usepackage{tikz}
\newtheorem{theorem}[equation]{Theorem}
\newtheorem{lemma}[equation]{Lemma}

\newtheorem{proposition}[equation]{Proposition}
\newtheorem{definition}[equation]{Definition}
\newtheorem{corollary}[equation]{Corollary}
\newtheorem{remark}[equation]{Remark}
\title{Neck Detection for Two-Convex Hypersurfaces Embedded in Euclidean Space undergoing Brendle-Huisken G-Flow}

\begin{document}
\author{Alexander Majchrowski} 
\date{}
\maketitle
\numberwithin{equation}{section}
\newcommand{\map}[2]{\,{:}\,#1\!\longrightarrow\!#2}
\newcommand{\defn}[1]{\textbf{#1}}
\newcommand{\newln}{\\&\quad\quad{}}
\renewcommand{\thefootnote}{\fnsymbol{footnote}}

	\begin{abstract}
	Recently Brendle-Huisken introduced a fully nonlinear flow $G$, \cite{brendle2015fully}. Their aim was to extend the surgery algorithm of Huisken-Sinestrari, \cite{huisken2009mean} into the Riemannian setting. The aim of this paper is to go through the details on how to perform neck detection for a closed, embedded hypersurface $M_0$ in $\mathbb{R}^{n+1}$ undergoing this $G$-flow.
	In order to do this we need to make some adjustments to Brendle and Huiskens gradient estimate from \cite{brendle2015fully}, after we have done this we can go on to argue as in \cite{huisken2009mean}, in order to classify two-convex surfaces undergoing $G$-flow.
	\end{abstract}

	\section{Introduction}
	
	In $1984$ Gerhard Huisken showed that for an $n$-dimensional ($n\geq2$) compact, uniformly convex surface without boundary smoothly embedded in Euclidean space, $\mathbb{R}^{n+1}$ undergoing mean curvature flow will converge to a round point in finite time. The normalised equation will converge to a sphere as $t\to\infty$, \cite{huisken1984flow}.
	
	In $1986$ he was able to prove the same result for a surface $M^n$, $n
	\geq2$ embedded in a Riemannian manifold $N^{n+1}$, however the requirements for this to happen were much stricter. We had to impose conditions on the sectional curvature of our surface, on the covariant derivative of the curvature tensor and the injectivity radius on our Riemannian manifold $N^{n+1}$, as well as having a stronger pinching condition on $\mathcal{M}_0$.
	
	The next natural step was to study $2$-convex hypersurfaces embedded in Euclidean space. This was done by Huisken-Sinestrari in 2009, \cite{huisken2009mean} using their surgery algorithm. They were able to show that any such surgery would terminate after a finite number of steps and that $\mathcal{M}_0$ was diffeomorphic to either $S^n$ or some finite connected sum of $S^{n-1}\times S^1$.
	
	Ideally one would then go on to prove a similar result for mean curvature flow of $2$-convex hypersurfaces embedded in a Riemannian Manifold. However in this setting $2$-convexity is not preserved by the flow. Inspired by Andrews work on harmonic mean curvature flow, \cite{andrews1993contraction}, Brendle-Huisken introduced the following flow which has the advantage of preserving $2$-convexity in the Riemannian setting.

	Fixing $n\geq 3$  consider a closed, embedded hypersurface $M_0$ in $\mathbb{R}^{n+1}$. $M_0$ is $\kappa$-$2$-convex if $\lambda_1+\lambda_2\geq 2\kappa$, where $\lambda_1\leq\cdots\leq\lambda_n$ denote the principal curvatures. We evolve $M_0$ with normal velocity
	\begin{align*}
	G_\kappa=\left(\sum_{i<j}\frac{1}{\lambda_i+\lambda_j-2\kappa}\right)^{-1}.
	\end{align*}
	We call this Brendle-Huisken G-flow, but will refer to it as G-flow in later sections.
	Brendle-Huisken were able to extend the surgery algorithm of Huisken-Sinestrari to this fully nonlinear flow in both the Euclidean setting and Riemannian setting. In order to do this they obtained a convexity estimate, cylindrical estimate and gradient estimate for this flow. 
	
   In the Euclidean setting we take $\kappa=0$. In order to argue as in Section 7, \cite{huisken2009mean} for this flow, adjustments have to be make for the gradient estimate. As we are not able to integrate these estimates in their current form to obtain results related to our backward parabolic neighbourhoods being surgery free as in \cite{huisken2009mean}. This is crucial in our proof of the neck detection lemma. Some small adjustments were also required in our proof of the neck detection lemma for this setting.
	
   After making these changes we can follow Section 7 of \cite{huisken2009mean} to obtain the other necessary results for when certain conditions in the neck detection lemma are not met. Firstly the we may not know that the backward parabolic neighbourhood about a point is surgery free, in this case we can obtain the required result as long as the curvature at our point is large enough compared to the curvature of the regions changed by previous surgeries. We must also deal with the case when $\frac{\lambda_1}{G}$ is not small, however the proof here does not rely on gradient estimates and is instead a general property of hypersurfaces as shown by Huisken-Sinestrari in \cite{huisken2009mean}.
   
   Putting this all together we obtain our main theorem.
   
   \begin{theorem}
   	Let $\mathcal{M}_t$ be a smooth Brendle-Huisken G-flow of a closed, compact $2$-convex hypersurface. Given our neck parameters, there exists a constant $G^*$ depending on $\mathcal{M}_0$ such that if $G_{\max}(t_0)\geq G^*$, then the hypersurface at time $t_0$ either contains an $(\epsilon,k,L)$-hypersurface neck or is convex.
   \end{theorem}
   
    We will assume that the reader is comfortable with the notation and definitions regarding necks. If not, please refer to \cite{hamilton1986four}, \cite{huisken2009mean} or \cite{majchrowski}.
   
   Then the arguments from Section 8 of \cite{huisken2009mean} carry over unchanged and we  can obtain results relating to the existence and classification of surgically modified flows.
   
    There are currently many open questions regarding this  $G$-flow. I am also currently working on adapting Joseph Lauer and John Head's method to show that this $2$-convex $G$-flow with surgeries converges to the weak solution of a level set flow for $G$ as we take our surgery parameter to infinity, \cite{majchrowskiconvergence}.
    
    Ancient solutions, solutions which are defined for all negative times have not been studied either, in the convex or $2$-convex case.
    
    Another option to consider would be to start putting details together on a suitable algorithm with estimates for dealing with $3$-convex flow with surgeries for the $G$-flow, which in this case would be
    \begin{align*}
    G_{\kappa}=\sum_{i<j<k}\left(\frac{1}{\lambda+i+\lambda_j+\lambda_k-3\kappa}\right)^{-1}.
    \end{align*} 
    
    This last problem would be hardest to work with, as no surgery algorithm for the $3$-convex case currently exists for mean curvature flow or Ricci flow.

	\vspace{5pt}	 
	\noindent\textbf{Acknowledgements}:
	I would like to thank Gerhard Huisken for introducing and advising me to work on this fully nonlinear flow $G$. 
	
	I would also like to thank Xu-Jia Wang for giving me an opportunity to present a talk on this topic at the Geometric and Nonlinear Partial Differential Equations Conference at Murramarang.
	
	Last but not least, I would like to thank my supervisor Zhou Zhang and the School of Mathematics at The University of Sydney for their help and support.
	\section{Evolution Equations and Necessary Estimates}
	
	In this section we go over some preliminary results obtained from \cite{brendle2015fully}, we will need the evolution equation for $G$, convexity estimate, cylindrical estimate, as well as our new gradient estimate which allows us to control the size of the curvature in the neighbourhood of a given point.
	
	Firstly we give some introductory results regarding  $G$-flow, stated by Brendle-Huisken in \cite{brendle2015fully}.
	
    \begin{proposition}\label{prop G}
    	Given $G$ as above we have the following properties:
    	\begin{enumerate}[(i)]
    		\item $G_{\kappa}\leq C_1H$, where $C_1=C_1(n)>0$.\\
    		\item $\frac{\partial G_{\kappa}}{\partial h_{ij}}\leq g_{ij}$, where $C_2=C_2(n)>0$.\\
    	\end{enumerate}
    \end{proposition}
    
    \begin{proof}
    	\begin{enumerate}[(i)]
    		\item Clear.
    		\item This is equivalent to observing $\frac{\partial G}{\partial \lambda_i}$ being bounded for each $i$..
    		\begin{align*}
    		\frac{\partial G}{\partial \lambda_i}&=\frac{\partial}{\partial \lambda_i}(\frac{1}{\lambda_1+\lambda_i}+\cdots+\frac{1}{\lambda_{i-1}+\lambda_i}+\frac{1}{\lambda_{i+1}+\lambda_i}+\cdots+\frac{1}{\lambda_n+\lambda_i})^{-1}\\
    		&=-1(\frac{1}{\lambda_1+\lambda_i}+\cdots+\frac{1}{\lambda_n+\lambda_i})^{-2}\times(\frac{-1}{(\lambda_1+\lambda_i)^2}+\cdots+\frac{-1}{(\lambda_1+\lambda_i)^2})\\
    		&\leq(\frac{1}{\lambda_1+\lambda_i}+\cdots+\frac{1}{\lambda_n+\lambda_i})^{-2}\times(\frac{1}{\lambda_1+\lambda_i}+\cdots+\frac{1}{\lambda_1+\lambda_i})^2\\
    		&=1
    		\end{align*}
    		where in the 2nd last step we have used $2$-convexity
    	\end{enumerate}
    \end{proof}			
    
    Now we are able to obtain the necessary evolution equations for $G$.		
	
	\begin{lemma}
		If $\mathcal{M}_t$ evolves by G-flow, the associated quantities above satisfy the following equations:
			\begin{enumerate}[(i)]
				\item $\frac{\partial}{\partial t}g_{ij}=-2Gh_{ij}$ \label{evolg}
				\item $\frac{\partial}{\partial t}g^{ij}=2Gh^{ij}$ \label{invevolg}
				\item $\frac{\partial}{\partial t}\nu=\nabla G$ \label{evolnu}
			     \item $\frac{\partial}{\partial t}h_{ij}=D_iD_j G-Gh_{il}g^{lm}h_{mj}$ \label{evolhij}
				\item $\frac{\partial}{\partial t}G=\frac{\partial G}{\partial h_{ij}}(D_iD_jG+h_{ik}h_{jk}G)$ \label{evolG}
				\item $\frac{\partial}{\partial t}H=\Delta|G|^2+\nabla|H|^2G.$ \label{evolH}
				\item $\frac{\partial }{\partial t}d\mu\leq -\frac{G^2}{C}$.
			\end{enumerate}\label{lemmaevol}
	\end{lemma}
	
	\begin{proof}
		\begin{enumerate}[(i)]
			\item 
			\begin{align*}
			 \frac{\partial}{\partial t}g_{ij}&= \frac{\partial}{\partial t}\left\langle\frac{\partial F}{\partial x_i},\frac{\partial F}{\partial x_j}\right\rangle\\\
			&=\left\langle\frac{\partial}{\partial x_i}(-G\nu), \frac{\partial F}{\partial x_j}\right\rangle+\left\langle \frac{\partial}{\partial x_j}(-G\nu), \frac{\partial F}{\partial x_i}\right\rangle\\\
			&=-G\left\langle\frac{\partial}{\partial x_i}\nu, \frac{\partial F}{\partial x_j}\right\rangle\-G\left\langle \frac{\partial F}{\partial x_i}, \frac{\partial}{\partial x_j}\nu\right\rangle\\\
			&=-2Gh_{ij}
			\end{align*}
			\item Obtained by differentiating $g_{il}g^{lj}=\delta_i^j$.
			\item 
			\begin{align*}
			 \frac{\partial}{\partial t}\nu&=\left\langle\frac{\partial}{\partial t}\nu,\frac{\partial F}{\partial x_i}\right\rangle\frac{\partial F}{\partial x_j}g^{ij}\\
			 &=-\left\langle\nu,\frac{\partial}{\partial t}\frac{\partial F}{\partial x_i})\right\rangle\frac{\partial F}{\partial x_j}g^{ij}\\
			 &=\left\langle\nu,\frac{\partial }{\partial x_i}(G\nu)\right\rangle\frac{\partial F}{\partial x_j}g^{ij}\\
			 &=\frac{\partial}{\partial x_i}G\frac{\partial F}{\partial x_j}g^{ij}=\nabla G.
			\end{align*}
			\item In this proof we will make use of the Gauss-Weingarten equations.
			\begin{align*}
			\frac{\partial}{\partial t} h_{ij}&=-\frac{\partial}{\partial t}\left\langle \frac{\partial^2 F}{\partial x_i\partial x_j},\nu \right\rangle\\
			&=\left\langle \frac{\partial^2}{\partial x_i\partial x_j}(G\nu),\nu \right\rangle-\left\langle \frac{\partial^2F}{\partial x_i\partial x_j},\frac{\partial}{\partial t} \nu \right\rangle\\
			&=\left\langle \frac{\partial^2}{\partial x_i\partial x_j} (G\nu),\nu \right\rangle-\left\langle\frac{\partial^2F}{\partial x_i\partial x_j}, \frac{\partial}{\partial x_l}G\frac{\partial F}{\partial x_m} g^{lm} \right\rangle\\
			&=\frac{\partial^2}{\partial x_i\partial x_j} G+ G\left\langle \frac{\partial}{\partial x_i}\left( \frac{\partial}{\partial x_j}\nu\right), \nu \right\rangle -\left\langle \frac{\partial^2F}{\partial x_i\partial x_j},\frac{\partial}{\partial x_l}G\frac{\partial F}{\partial x_m} g^{lm} \right\rangle\\
			&=\frac{\partial^2}{\partial x_i\partial x_j} G+G\left\langle \frac{\partial}{\partial x_i}\left(h_{jl}g^{lm}\frac{\partial F}{\partial x_m}\right),\nu \right\rangle\\
			&-\left\langle \Gamma_{ij}^k\frac{\partial F}{\partial x_k}-h_{ij}\nu, \frac{\partial}{\partial x_l}G\frac{\partial F}{\partial x_m} g^{lm} \right\rangle \\
			&=\frac{\partial^2}{\partial x_i\partial x_j}G-\Gamma_{ij}^k \frac{\partial}{\partial x_k}G+Gh_{jm}g^{ml}\left\langle\gamma_{il}^{o}\frac{\partial F}{\partial x_o}-h_{il}\nu,\nu \right\rangle\\
			&=D_iD_j G - Gh_{il}g^{lm}h_{mj}.
			\end{align*}
			\item  Here we just use the fact that $\frac{\partial G}{\partial t}=\frac{\partial G}{\partial h_i^j}\frac{\partial}{\partial t}h_i^j$ as well as part (ii) and (iv).
			\item 
			\begin{align*}
			\frac{\partial}{\partial t} H&= \frac{\partial}{\partial t}(h_{ij}g^{ij})\\
			&=g^{ij} \frac{\partial}{\partial t} h_{ij}+2Gh^{ij}h_{ij}\\
			&=g^{ij}(D_iD_jG-Gh_{il}g^{lm}h_mj)+2Gg^{ik}g^{jl}h_{kl}h_{ij}\\
			&=\Delta G+|h|^2G.
			\end{align*}
			\item $\frac{\partial }{\partial t}d\mu=\frac{1}{2}\sqrt{\det g_{ij}}tr(-2Gh_{ij})=-GH\mu$. Using part (i) of Proposition \ref{prop G} the result follows.
		\end{enumerate}
	\end{proof}
	
	The convexity estimate (\cite{brendle2015fully}, Corollary 7.7) is necessary in order to know that the nearly singular regions of the surface become asymptotically convex as a singular time is approached.
	
	\begin{theorem}[Convexity Estimate]\label{convexityG}
		Suppose that $\mathcal{M}_t$, $t\in[0,T)$ is a surgically modified G-flow starting from a closed, embedded, $2$-convex hypersurface $\mathcal{M}_0$ then for any $\delta>0$
		\begin{align*}
		\lambda_1\geq \delta G- C
		\end{align*}
		where $C$ is a positive constant that depends only on $\delta,n$ and $T$.
	\end{theorem}
	
	Next we need a cylindrical estimate (\cite{brendle2015fully}, Theorem 3.1) which implies that at points where $\lambda_1$ is small, we have curvature close to the curvature of a cylinder. 
	
	\begin{theorem}[Cylindrical Estimate]\label{cylindricalG}
		Let $\mathcal{M}_t$ be a family of closed, two-convex hypersurfaces moving with speed $G$, then $\forall\eta>0$ there exists a constant $C=C(\delta,T,n)>0$ such that
		\begin{align*}
		H\leq\frac{(n-1)^2(n+2)}{4}(1+\delta)G+C_{\eta,T}.
		\end{align*}
	\end{theorem}
	
	The following is the gradient estimate taken from \cite{brendle2015fully}, Theorem 7.12.  
	
	\begin{theorem}[Gradient Estimate]\label{gradestG}
		Far a closed, embedded, two-convex hypersurface $M_0=\delta\Omega_0$. We can find a constant $G_\#$, depending only on $M_0$ such that the following holds. Suppose that $\Omega_t$, $t\in[0,T)$, is a one-parameter family of sooth open domains with the property that the hypersurfaces $M_t=\partial\Omega_t$ form a surgically modified flow starting from $M_0$ with surgery scale $G_*\geq G_\#$. Then we have
		\begin{equation}
		\alpha^{2}G^{-2}|\nabla h|+\alpha^3G^{-3}|\nabla^3 h|\leq \Lambda
		\end{equation}
		for all points in spacetime satisfying $G\geq G_\#$. Here $\alpha=\alpha(T,n)$ is the constant in Proposition 7.8 (\cite{brendle2015fully}), and $\Lambda=\Lambda(T,n)$ is the constant appearing in Corollary 7.11 \cite{brendle2015fully}.
	\end{theorem}
	
	Now we want to modify this gradient estimate using the following lemma, this will allow us to integrate our gradient estimates and obtain necessary results related to the backward parabolic neighbourhood as done in \cite{huisken2009mean}.
	
	\begin{lemma}\label{gradestlemma}
		The inequalities $\alpha^2G^{-2}|\nabla h|\leq C$ and $\alpha^3G^{-3}|\nabla^2 h|\leq C$ from Theorem \ref{gradestG} imply the following,
		\begin{enumerate}[(i)]
			\item $K_1|\nabla G|\leq|\nabla h|$
			\item $K_2|\partial_t G|\leq|\nabla^2 h|$
		\end{enumerate}
		for some constants $K_1,K_2$.
	\end{lemma}
	
	\begin{proof}
		For the $n$-dimensional case we look at the following. We know that
		\begin{align*}
		\det(\lambda I -h_{ij})&=(\lambda-\lambda_1)(\lambda-\lambda_2)\cdots(\lambda-\lambda_n) \\
		\Rightarrow 
		\begin{bmatrix}
		\lambda-h_{11}   &\dots & -h_{1n} \\
		\vdots &\ddots &\vdots\\
		-h_{n1}     & \dots & \lambda-h_{nn}
		\end{bmatrix}&=\lambda^n-\lambda^{n-1}(\lambda_1+\cdots+\lambda^n)+\cdots-\lambda_1\dots\lambda_n.	
		\end{align*}

		Here we will need to introduce some notation. Let $Q_k(h_{ij})$ denote a $k$-degree polynomial in terms of of $h_{ij}$'s, no lower degree can appear. Using a degree argument and equating terms on either side we will obtain
		\begin{align}
		\lambda^pQ_{n-p}(h_{ij})=\lambda^p\lambda_j^{n-p}.\label{degreecase}
		\end{align}
		This guarantees that we can rewrite our principal curvature values in terms of the second fundamental form.
		Rewriting $G$ as follows,
		\begin{align}
		G=\frac{\prod_{i<j}(\lambda_i+\lambda_j)}{\sum_{i<j}\frac{1}{\lambda_i+\lambda_j}\prod_{i<j}(\lambda_i+\lambda_j)}\label{newG}
		\end{align}
		applying the result of (\ref{degreecase}) to our rewritten $G$ we see that we can write out the $\lambda^p$ using our $h_{ij}^p$ terms, 
		\begin{align*}
		G=\frac{Q_n(h_{ij})}{Q_{n-1}(h_{ij})}
		\end{align*}
		
		Moreover from our definition of $G$, (\ref{newG}) and $2$-convexity, we can see that 
		\begin{align}
		\frac{\lambda_1+\lambda_2}{n}\leq G\leq\lambda_1+\lambda_2\label{Gineqs}
		\end{align}
		lastly from Proposition 2.4 in \cite{brendle2015fully} we know that $H\leq \beta G$ for some constant $C$. This tells us that
		\begin{align}
		&\lambda_1+\cdots+\lambda_n=H\leq \beta G \nonumber\\
		&\Rightarrow \lambda_i\leq \beta_1G\nonumber\\
		&\Rightarrow h_{ij}\leq \beta_2G \;\text{for some constants}\; \beta_1,\beta_2\label{betaG}.
		\end{align}
		
		The last step is clear as we know $(h_{ij})=O(\lambda_i)O^T$, where $O$ is an orthonormal matrix.
		So,
		\begin{align*}
		\nabla G&=\frac{Q_1(\nabla h_{ij})Q_{n-1}( h_{ij})Q_{n-1}(h_{ij})-Q_1(\nabla h_{ij})Q_{n-2}(h_{ij})Q_n(h_{ij})}{Q_{2n-2}( h_{ij})}\\
		&=\frac{Q_1(\nabla h_{ij})Q_{2n-2}( h_{ij})}{Q_{2n-2}(h_{ij})}\\
		&\leq\frac{\beta_3G^{n-2}Q_1(\nabla h_{ij}) }{\beta_4 G^{2n-2}}\;\text{by}\;(\ref{betaG})\;\text{and}\;(\ref{Gineqs})\;\text{where}\;\beta_3,\beta_4\;\text{are constants.}\\
		&\Rightarrow\;|\nabla G|\leq|\frac{\beta G^2}{\alpha^2}|
		\end{align*}
		were in the last line we have applied Proposition \ref{gradestG}. This proves $(i)$.
		
		Now we  prove part $(ii)$.  From \ref{lemmaevol} (v) we know that
		\begin{equation}
		\partial_tG=\frac{\partial G}{\partial h_{ij}}(\nabla_i\nabla_j G-h_{ij}h_{jk}G).
		\end{equation}
		We can control $\frac{\partial G}{\partial h_{ij}}$ using part (ii) of Proposition \ref{prop G}.
		
		Next by applying (\ref{betaG}) we can bound the $h_{ij}h_{jk}G$ term by $\beta G^3$ for some constant $\beta$.	
		
		Lastly $\nabla_i\nabla_j G$ will give terms of the form $\nabla^2h_{ij}$ and $\nabla h_{ij}G$. Using our result from part $(i)$ as well Proposition \ref{gradestG} we see that $|\partial_t G|\leq |KG^3|$ for some constant $K$. This completes the proof of part $(ii)$.
	\end{proof}
	
	By apply Lemma \ref{gradestlemma} to Theorem \ref{gradestG} we can obtain our new gradient estimate stated below.
	
	\begin{theorem}\label{Ggradests}
		Let $M$ be a fully nonlinear flow with surgeries. Then we can find $c^{\#}>0, G^{\#}>0$ such that for all $p\in M$ and $t>0$,
		\begin{align}
		G(p,t)>G^{\#}>0\;\Rightarrow\;|\nabla G(p,t)|\leq c^{\#}G^2(p,t),\;|\partial_t G(p,t)|\leq c^{\#}G^3(p,t),
		\end{align}
		where $c^{\#}$ only depends on the dimension of $n,T$.
	\end{theorem}
	
	These estimates allow us to control the size of the curvature in a neighbourhood of a given point.  
	
	\section{The Neck Detection Lemma}

	Using our new gradient estimate we will be able to obtain a result relating to the size of the curvature in a neighbourhood of a given point.
	
	\begin{lemma}\label{6.6HSforG}
		Let $F:\mathcal{M}\to\mathbb{R}^{n+1}$ be an $n$-dimensional immersed surface. Suppose that there are $c^\#, G^\#>0$ such that $|\nabla G(p)|\leq c^\# G^2(p)$ for any $p\in\mathcal{M}$ such that $G(p)\geq G^\#$. Let $p_0\in\mathcal{M}$ satisfy $G(p_0)\geq\gamma G^\#$ for some $\gamma>1$. Then
		\begin{align*}
		G(q)&\geq\frac{G(p_0)}{1+c^\#d(p_0,q)G(p_0)}\geq \frac{G(p_0)}{\gamma}\;\text{for all $q$}\\
		\text{s.t.} \; d(p_0,q)&\leq\frac{\gamma-1}{c^\#}\frac{1}{G(p_0)}.
		\end{align*}
	\end{lemma}
	
	\begin{proof}
		Consider points $q\in\mathcal{M}$ such that $G(q)<\frac{G(p_0)}{\gamma}$.  Take $q_0$ to be a point with this property with minimal distance from $p$, and set $d_0=d(p_0,q_0)H(p_0)$ and $\theta_0=\min\{d_0,\frac{\gamma-1}{c^\#}\}$. Now pick any point $q\in\mathcal{M}$ with $d(p_0,q)\leq\frac{\theta_0}{H(p_0)}$, and let $\xi:[0,d(p_0,q)]\to\mathcal{M}$ be a geodesic from $p_0$ to $q$.
		
		Then from our definition of $\theta_0$ it follows that $G(\xi(s))\geq\frac{G(p_0)}{\gamma}\geq G^\#$ for any $s\in[0,d(p_0,q)]$. Then we can apply Lemma \ref{gradestG} to obtain $|\nabla G(\xi(s))|\leq c^\#G^2(\xi(s))$ and 
		\begin{align*}
		\frac{d}{ds}G(\xi(s))\geq -c^\#G^2(\xi(s))
		\end{align*}for all $s\in[0,d(p_0,q)]$ since it is a geodesic. Integrating this inequality we obtain
		\begin{align*}
		G(xi(s))\geq\frac{G(p_0)}{1+c^\#sG(p_0)},\;s\in[0,d(p_0,q)],
		\end{align*}
		which implies
		\begin{align}\label{G6.6}
		G(q)\geq\frac{G(p_0)}{1+c^\#d(p_0,q)G(p_0)}\geq\frac{G(p_0)}{1+c^\#\theta_0}.
		\end{align}
		This holds for all $q$ such that $d(p_0,q)\leq\frac{\theta_0}{G(p_0)}$. Now suppose $d_0<\frac{\gamma-1}{c^\#}$, then $d_0=\theta_0$ and (\ref{G6.6}) holds for $q=q_0$. But that implies  $G(q_0)>\frac{G(p_0)}{\gamma}$ which is a contradiction. Therefore $d_0\geq\frac{\gamma-1}{c^\#}$, which implies $\theta_0=\frac{\gamma-1}{c^\#}$, which makes (\ref{G6.6}) our assertion.
		
		In the case where $G(q)\geq\frac{G(p_0)}{\gamma}$ for all $q\in\mathcal{M}$, then we have $|\nabla G|\leq c^\#G^2$ everywhere, and our result follows more directly from the same argument.
	\end{proof}
	    
	    Next we introduce a backward parabolic neighbourhood. This will be essential in dealing with necks.
	    
	  \begin{definition}
	    	Given $t,\theta$ such that $0\leq t-\theta<t\leq T_0$, we define the backward parabolic neighbourhood of $(p,t)$ by,
	    	\begin{align}
	    	\mathcal{P}(p,t,r,\theta)=\{(q,s)|q\in\mathcal{B}_{g(t)}(p,r),s\in[t-\theta,t]\}. \label{backwardpara}
	    	\end{align}
	    	where $\mathcal{B}_{g(t)}(p,r)\subset\mathcal{M}$ is the closed ball of radius $r$ w.r.t. the metric $g(t)$.
	  \end{definition}
	    
	Before we can now apply Lemma 7.2 as in \cite{huisken2009mean}. We need to also define $\hat{r}_G=\frac{(n-1)(n-2)}{2G}$ and $\mathcal{\hat{P}}_G=\mathcal{P}(p,t,\hat{r}_G(p,t)L,\hat{r}^2(p,t\theta)$, we can observe that if $(p,t)$ lies on a neck then $\hat{r}_G(p,t)$ is approximately equal to the radius of the neck.
	
	\begin{lemma}\label{Ghash}
		Let $c^{\#}$ and $G^{\#}$ be the constant from the above Corollary Define $d^{\#}=(2(n-1)^2(n-2)^2c^{\#})^{-1}$. Then the following properties hold.
		\begin{enumerate}[(i)]
			\item Let $(p,t)$ satisfy $G(p,t)\geq 2G^\#$. Then, given any $r,\theta\in(0,d^\#]$ such that $\mathcal{\hat{P}}_G(p,t,r,\theta)$ does not contain surgeries, we have
			\begin{align}
			\frac{G(p,t)}{2}\leq G(q,s)\leq G(p,t)
			\end{align}
			for all $(q,s)\in\mathcal{\hat{P}}_G(p,t,r,\theta)$.
			\item Suppose that, for any surgery performed at time less than $t$, the regions modified b surgery have mean curvature less than $K$, for some $K\geq G^\#$. Let $(p,t)$ satisfy $G(p,t)\geq 2K$. Then the parabolic neighbourhood
			\begin{align}
			\mathcal{P}(p,t,\frac{1}{8c^\#K},\frac{1}{8c^\#K^2}) \label{surgfreeG}
			\end{align}  does not contain surgeries. In particular, the neighbourhood $\hat{\mathcal{P}}_G(p,t,d^\#,d^\#)$ does not contain surgeries and all points $(q,s)$ contained there satisfy (i).
		\end{enumerate}
	\end{lemma}
	
	\begin{proof}
		Firstly we prove part (ii).
		
		Suppose the neighbourhood of in (\ref{surgfreeG}) is modified by surgeries. Take a point $(q,s)$ which is modified by surgery, with $s$ the maximal time at which we can find such a point. Then by assumption we have $G(q,s+)\leq K$. Integrating the estimate on $\partial_tG$ from Theorem \ref{Ggradests},
		\begin{align*}
		&\int_{s}^{t}\frac{\partial G}{G^3}\leq \int_{s}^{t} c^\# \partial t\\
		-\frac{1}{G^2(q,t)}&+\frac{1}{G^2(q,s)}\leq c^\#(t-s)\\
		\frac{1}{G^2(q,t)}&\geq\frac{1}{G^2(q,s)}-2c^\#(t-s)\geq \frac{3}{4K^2}
		\end{align*}
		where in the last line we used our assumption on $H(q,s)$ and that $t-s\leq\frac{1}{8c^\#K^2}$.
		Then we integrate along a geodesic from $q$ to $p$ at time $t$ and use the estimate on $\nabla G$ 
		\begin{align*}
		\frac{1}{G(p,t)}\geq\frac{1}{G(q,t)}-c^\#d_{g(t)}(p,q)\geq\frac{\sqrt{3}-1/4}{2K}>\frac{1}{2K},
		\end{align*}
		where in the last line we used our estimate on $G^2(q,t)$ and that $d_{g(t)}\leq\frac{1}{8c^\#K}$. This contradicts our assumption  that $G(p,t)\geq 2K$. In this argument we have had to assume that $G\geq G^\#$ along the integration paths in order to apply the results of Theorem \ref{Ggradests}. If this is not true, we can choose the last point along the path with $G\leq G^\#$ and integrate from that point on, obtaining a contradiction in the same way. Now we use the definition of $d^\#$ to see that $\hat{\mathcal{P}}_G(p,t,d^\#,d^\#)$ is contained in the neighbourhood (\ref{surgfreeG}),
		\begin{align*}
		&\frac{(n-1)(n-2)}{2G(p,t)}d^\#\leq \frac{1}{8K(n-1)(n-2)c^\#}\leq\frac{1}{8c^\#K}\\
		\text{and}\;&\frac{(n-1)^2(n-2)^2}{2G(p,t)^2}d^\#\leq\frac{1}{16K^2c^\#}\leq\frac{1}{8c^\#K^2}.
		\end{align*}	
		Therefore $\hat{\mathcal{P}}_G(p,t,d^\#,d^\#)$ does not contain surgeries and part $(i)$ can be applied to this neighbourhood.
		
		To prove part (i), we integrate the same inequalities and use the assumption that $\mathcal{\hat{P}}_G$ is surgery free.
	\end{proof}

	\begin{lemma}[Neck Detection Lemma]\label{neckdetectionG}
		Let $\mathcal{M}_t, t\in[0,T)$ be the G-flow with surgeries, starting from an initial manifold $\mathcal{M}_0$. Let $\epsilon,\theta,L>0$, and $k\geq k_0$ be given (where $k_0\geq2$ is the parameter measuring the regularity of the necks where surgeries are performed). Then we can find $\eta_0, G_0$ with the following property. Suppose that $p_0\in\mathcal{M}_0$ and $t_0\in[0,T)$ are such that:
		\begin{enumerate}[(ND1)]
			\item $G(p_0,t_0)\geq G_0,\frac{\lambda_1(p_0,t_0)}{G(p_0,t_0)}\leq\eta_0$
			\item The neighbourhood $\hat{\mathcal{P}}(p_0,t_0,L,\theta)$ does not contain surgeries.
		\end{enumerate}
		Then
		\begin{enumerate}[(i)]
			\item The neighbourhood $\hat{\mathcal{P}}(p_0,t_0,L,\theta)$ is an $(\epsilon,k_0-1,L,\theta)$-shrinking curvature neck;
			\item The neighbourhood $\hat{\mathcal{P}}(p_0,t_0,L-1,\theta/2)$ is an $(\epsilon,k,L-1,\theta/2)$-shrinking curvature neck.
		\end{enumerate}
		With constants $\eta_0(\alpha,\epsilon,k,L,\theta)$.
	\end{lemma}
	
	\begin{proof}
		Here we argue by contradiction. Suppose that for some values of $\epsilon,L,\theta$ the conclusion doesn't hold. No matter how we pick $\eta_0$ or $G^*$. We take a sequence $\{M_t^j\}_{j\geq1}$ of solutions to the flow. A sequence of times $t_j$ and points $p_j$ such that $,H_j,G_j,\lambda_{1,j}$ denote the values of $H,G$ and $\lambda_1$ at $F_j(p_j,t_j)\in M_{t_j}^j$ satisfy,
		
		\begin{enumerate}[(i)]
			\item $\lambda_1+\lambda_2\geq\alpha_0 G$, this is clear from (\ref{Gineqs}).
			\item parabolic nbhd not changed by surgery
			\item $G_j\to\infty,\;\frac{\lambda_{1,j}}{G_j}\to0$ as $j\to\infty$.
			\item $(p_j,t_j)$ doesn't lie at the centre of an $(\epsilon,k_0-1,L,\theta)$-shrinking neck.
		\end{enumerate}
		$G_j\to\infty$ since curvature of the flows uniformly bounded at $t=0$. We now continue with a parabolic rescaling such that, $H(p_j,t_j)=1$ and $(p_j,t_j)$ is translated to the origin $0\in\mathbb{R}^{n+1}$ and $t_j$ becomes $0$, we define it as
		\begin{align*}
		\tilde{F}(p,\tau)=\frac{1}{H_j}[F_j(p,\hat{r}^2\tau+t_j)-F(p_j,t_j)]
		\end{align*}
		
		denoting by $\bar{\mathcal{M}}_\tau^j$ our rescaled surface.
		
		Our gradient estimates give bounds on $|A|$ and derivatives up to $k_0$.  This tells us we have a limit surface to the flow $\tilde{M}_\tau^{\infty}$.
		
		Passing to the limit in the convexity estimate from Lemma ($\ref{convexityG}$) we find that $\tilde{\lambda}_1\geq 0$ which implies that $\tilde{\lambda}_i\geq0$ for all $i=1,...n$. On the other hand we know that $\lambda_1+\lambda_2\geq0$ so $\tilde{\lambda}_j>0$ for $j\geq2$ and $\tilde{S}_i>0$ for all $i\leq n-1$, where $S_i$ is defined in \cite{huisken1999convexity}. This together with $(iii)$ tells us $\tilde{\lambda}_1(0,0)=0$.

		Looking at the quotient $\tilde{Q}_n=\frac{\tilde{S}_n}{\tilde{S}_{n-1}}$ we know it is non negative everywhere. Arguing as in \cite{huisken1999convexity} and using the strong maximum principle, we know if it is positive somewhere in the interior of $\tilde{P}^{\infty}(0,0,d,d)$, it is positive everywhere at all later times. But $\tilde{Q}_n(0,0)=0$, showing that $\tilde{\lambda}_1\equiv0$ in this set.

		Scaling by $H$ I know that my principal curvatures on a cylinder are $\lambda_1=0$ and $\lambda_j=\frac{1}{n-1}$ for all $j\geq2$. 
		\begin{align*}
		\Rightarrow H=1\;\Rightarrow G&=((n-1)(n-1)+\frac{(n-1)(n-2)}{2}\frac{(n-1)}{2})^{-1}\\
		&=(\frac{(n-1)^2(n+2)}{4})^{-1}.
		\end{align*}
		On a cylinder $H-\frac{((n-1)^2(n+2)}{4}G=0$. I know passing to the limit in my cylindrical estimate yields $H-\frac{(n-1)^2(n+2)}{4}G\leq0\;\Rightarrow G\geq\frac{4}{(n-1)^2(n+2)}$, we want to see that $\frac{4}{(n-1)^2(n+2)}$ is the maximum value $G$ can attain.
		
		We want to see, $G(0,a_1,...,a_{n-1})\leq G(0,\frac{1}{n-1},...,\frac{1}{n-1})$ with equality when the $a_i$'s are equal.
		
		Picking from $(0,a_1,...,a_{n-1})$ we know,
		\begin{align*}
		G^{-1}=\frac{1}{a_1}+\cdots+\frac{1}{a_{n-1}}+\sum\frac{1}{a_i+a_j}+\lambda(1-a_1+\cdots+a_{n-1})
		\end{align*}
		where $\lambda$ is the Lagrange multiplier with the constraint that the sum of our $a_i's=1$.
		Taking partial derivatives we see that for any two $a_i,a_k$,
		\begin{align*}
		\frac{\partial G^{-1}}{\partial a_i}&=-\frac{1}{a_i^2}-\sum_{j\neq i}\frac{1}{(a_i+a_j)^2}-\lambda=0\\
		\frac{\partial G^{-1}}{\partial a_k}&=-\frac{1}{a_k^2}-\sum_{j\neq k}\frac{1}{(a_k+a_j)^2}-\lambda=0\\
		\Rightarrow -\frac{1}{a_i^2}&+\frac{1}{a_k^2}-\sum_{j\neq i,k}\frac{1}{(a_i+a_j)^2}+\sum_{j\neq i,k}\frac{1}{(a_k+a_j)^2}=0\\
		\end{align*}
		$\Rightarrow \;a_i=a_k$ and all $a_i's$ are equal. Therefore we have a maximum when they are equal, on the boundary we will have a minimum. This tells me that $G=C(n)H$ for some constant $C=C(n)$, so on the limit I have that my fully nonlinear $G$-flow is the same as mean curvature flow. This allows me to argue as in [\cite{huisken2009mean}, Lemma 7.4] and [\cite{huisken1993local}, Theorem 5.1] to obtain that $\tilde{\mathcal{M}}_t^{\infty}$ on $\tilde{\mathcal{P}}^{\infty)}(0,0,d,d)$ is a portion of a shrinking cylinder.

		Now we can iterate the procedure to that the whole neighbourhoods $\bar{\mathcal{P}}^{\infty}_j(0,0,L,\theta)$ of the rescaled flow converge to a cylinder. From the first step we know that, for $j$ large enough, the curvature on $\bar{\mathcal{P}}_j^{\infty}(0,0,d,d)$ is close to the curvature of a unit cylinder. Then, using the gradient estimates we have uniform bounds on $\bar{G}_j$ also on some larger neighbourhoods, i.e. $\bar{\mathcal{P}}_j^{\infty}(0,0,2d,2d)$, we can repeat the previous argument to prove convergence to a cylinder there. After a finite number of iterations we can obtain convergence of the neighbourhoods $\bar{\mathcal{P}}_j^{\infty}(0,0,L,\theta)$. The immersions converge in the $C^{k_0-1}$-norm and this ensures that for $j$ large enough, the neighbourhoods are $(\epsilon,k_0-1,L,\theta)$-shrinking necks. This contradicts assumption $(iv)$ and proves part $(i)$ of the lemma.

				To prove part $(ii)$ of the lemma we argue in a similar fashion. Again we argue by contradiction and take a sequence of rescaled flows. Consider the smaller parabolic neighbourhoods $\bar{\mathcal{P}}(0,0,L-1,\frac{\theta}{2})$ and apply interior regularity results from \cite{ecker1991interior} to find bounds in the $C^{k+1}$ norm as well. This yields compactness in the $C^k$-norm, which yields the desired result.
	\end{proof}
	
	\begin{remark}
		Part $(i)$ of the lemma concerns the whole parabolic neighbourhood which is surgery free, but can be arbitrarily close to surgery, the points of the neighbourhood are even allowed to be modified by a surgery at the initial time $t_0-\theta r_0^2$. Therefore the description goes up to $k_0-1$ derivatives. Part $(ii)$ is concerned with a smaller parabolic neighbourhood, where we can use interior parabolic regularity and as many derivatives as we wish.
	\end{remark}
	
	\begin{corollary}
		Given $\epsilon,\theta>0, L\geq 10$ and $k>0$ an integer, we can find $\eta_0, H_0>0$ such that the following holds. Let $p_0,t_0$ satisfy $(ND1)$ and $(ND2)$ of Lemma \ref{neckdetectionG}. Then
		\begin{enumerate}[(i)]
			\item The point $(p_0,t_0)$ lies at the centre of a cylindrical graph of length $2(L-2)$ and $C^{k+2}$-norm less than $\epsilon$;
			\item The point $(p_0,t_0)$ lies at the centre of a normal $(\epsilon,k,L-2)$-hypersurface neck.
		\end{enumerate}
	\end{corollary}
	
	\begin{proof}
		Using Proposition 3.5 and Theorem 3.14 from \cite{huisken2009mean} both properties are true is $(p_0,t_0)$ lies at the centre of an $(\epsilon',k',L-1)$-curvature neck for suitable $\epsilon',k'$. Therefore it is clear that the properties hold if $(p_0,t_0)$ lies at the centre of an $(\epsilon',k',L-1,\frac{\theta}{2})$-shrinking curvature neck. Thus is suffices to apply part $(ii)$ of the neck detection lemma with parameters $(\epsilon',k',L,\theta)$ and use the corresponding values of $\eta_0, G_0$.
	\end{proof}
	
	The following lemma, is an immediate consequence of the neck detection lemma. It shows that the shrinking curvature neck given by that lemma can be represented at each time as a hypersurface neck. This can be done at all times, including the initial time which could coincide with a surgery time.
	
	For this we will introduce the following. We say that a point $(p,t)$ lies at the centre of a neck if $p\in\mathcal{M}$ lies at the centre of a neck with respect to the immersion $F(\cdot,t)$. We introduce a time-dependent version of the notion of curvature neck. For $s\leq0$ we set,
	\begin{align}\label{rho}
	\rho(r,s)=\sqrt{r^2-2(n-1)s}
	\end{align}
	so that $\rho(r,s)$ is the radius at time $s$ of a standard $n$-dimensional cylinder evolving by the fully nonlinear flow having radius $r$ at time $s=0$. Moreover we have
	\begin{align}\label{curvneckrho}
	r\leq\rho(r,s)\leq2r\;\text{for all}\;s\in[d^\#r^2,0]
	\end{align}
	otherwise the cylinder would violate Lemma \ref{Ghash}.
	\begin{lemma}\label{7.9Huisken}
		In Lemma \ref{neckdetectionG}, we can choose the constants $\eta_0,h_0$ so that the additional following property holds. Suppose that $L\geq10$ and that $\theta\leq d^\#$. Denote
		\begin{align*}
		r_0=\frac{(n-1)(n-2)}{2G(p,t)},\;B_0=B_{g(t_0)}(p_0,r_0L).
		\end{align*}
		Then for any $t\in[t_0-\theta r_0^2,t_0]$, the point $(p_0,t)$ lies at the centre of an $(\epsilon,k_0-1)$-hypersurface neck $N_t\subset B_0$, satisfying the following properties:
		\begin{enumerate}[(i)]
			\item The mean radius $r(z)$ of every cross section of $N_t$ is equal to $\rho(r_0,t-t_0)(1+O(\epsilon))$;
			\item The length of $N_t$ is at least $L-2$;
			\item There exists a unit vector $\omega\in\mathbb{R}^{n+1}$ such that $|\nu(p,t)\cdot\omega|\leq\epsilon$ for any $p\in N_t$.
		\end{enumerate}
	\end{lemma}
	
	\begin{proof}
		Proved in the same way as part $(i)$ of Lemma \ref{neckdetectionG}. By contradiction, we take a suitable $\eta_0,G_0$, then our parabolic neighbourhood is as close as we wish to an exact cylinder evolving by the fully nonlinear flow over the same time interval. The cylinder has radius $r_0$ at the final time, hence it has radius $\rho(r_0,t-t_0)$ at previous times.
		
		At the final time, $C_{t_0}$ is a neighbourhood of radius $r_0L$ of $p_0$. Let $B_L\subset C$ be the set of points of $C$ having intrinsic distance less than $L$ from $p_0$. Clearly, $B_L$ cannot be written in the form $S^{n-1}\times[a,b]$ for any $a,b,$. However, it is easy to see that for $L\geq(\pi^2+1)/2$, then $S^{n-1}\times[-(L-1),(L-1)]\subset B_L\subset S^{n-1}\times[-L,L]$. Using this logic, we can see $C_{t_0)}$ contains a sub-cylinder of length $2(L-1)$. The same sub-cylinder is contained in $C_t$ for $t<t_0$; however since the scaling factor is given by $\rho(r_0,t-t_0)$ rather than $r_0$, then length of the sub-cylinder becomes $\frac{2r_0(L-1)}{\rho(r_0,t-t_0)}$. Recalling $(\ref{curvneckrho})$, we see the sub-cylinder has length at least $2(L-1)$ for the times under consideration. Since we can make our parabolic neighbourhood as close as we wish in the $(k_0-1)$-norm to the cylinder $C_t$ we can find a geometric neck parametrizing the part of the neighbourhood corresponding to the sub-cylinder found above, and this neck will satisfy properties $(i)$ and $(ii)$. Property $(iii)$ follows from choosing $\omega$ to be the axis of our cylinder $C_t$.
	\end{proof}

	Just as before in the mean curvature flow case, we rely on $(ND2)$ and have seen it is crucial in the proof of the neck detection lemma. The next result shows that $(ND2)$ follows from the other assumptions of the neck detection lemma provided the curvature at $(p_0,t_0)$ is large enough compared to the curvature of the regions changed by the previous surgeries.
	
	\begin{lemma}
		Consider a flow with surgeries satisfying the same assumptions of Lemma \ref{neckdetectionG}. Let $d^\#$ be as before and let $\epsilon,k,L,\theta$ be given with $\theta< d^\#$. Then we can find $\eta_0,G_0$ with the following property. Let $(p_0,t_0)$ be any point satisfying
		\begin{align*}
		G(p_0,t_0)\geq\max\{G_0,5k\},\;\frac{\lambda_1(p_0,t_0)}{G(p_0,t_0)}\leq\eta_0
		\end{align*}
		where $K$ is the maximum of the curvature at the points changed in the surgeries at times before $t_0$. Then $(p_0,t_0)$ satisfies hypothesis $(ND2)$ and the conclusions $(i)-(ii)$ of Lemma \ref{neckdetectionG}. In addition, the neighbourhood
		\begin{align*}
		\mathcal{P}(p_0,t_0,\frac{(n-1)(n-2)}{2G(p_0,t_0)}L,\frac{(n-1)^2(n-2)^2}{4K^2}d^\#),
		\end{align*}
		which is larger in time than $(ND2)$ does not contain surgeries. 
	\end{lemma}
	
	\begin{proof}
		Let $\epsilon,k,L,\theta$ be given, with $\theta\leq d^\#$. The constants $\eta_0,G_0$ depend continuously on the parameters $L,\theta$ measuring the size of the parabolic neighbourhood. Thus, if $L_2>L_1$ and $\theta_2>\theta_1>0$ it is possible to find $\eta_0,H_0$ which apply to any $L\in[L_1,L_2]$ and $\theta\in[\theta_1,\theta_2]$. Thus we can find values $\eta_0,G_0$ such that the conclusions of the neck detection lemma hold for our choice of $(\epsilon,k,L,\theta)$, and also if we replace $L$ with any $L'\in[d^\#,L]$. In addition, we can also assume that $G_0\geq 2G^\#$. We claim that such values of $\eta_0,G_0$ satisfy the conclusions of the present lemma.
		
		By our choice of $\eta_0,G_0$, the conclusions don't hold only if $(ND2)$ is not satisfied, i.e.$\hat{\mathcal{P}}(p_0,t_0,L,\theta)$ contains surgeries.
		
		By Lemma \ref{Ghash}, at least the neighbourhood $\hat{\mathcal{P}}(p_0,t_0,d^\#,0)$ does not contain surgeries. Therefore, if $(ND2)$ s violated, there exists a maximal $L'\in[d^\#,L]$ such that $\hat{\mathcal{P}}(p_0,t_0,L,\theta)$ does not contain surgeries. We apply the neck detection lemma to this neighbourhood and deduce that it is an $(\epsilon,k_0-1,L',\theta)$-shrinking neck. In particular $G(p_0,t_0)(1+O(\epsilon))\geq 4K$ for all $p$ such that $d_{g(t_0)}(p_0,p)\leq\hat{r}(p_0,t_0)d^\#$. But then Lemma \ref{Ghash} shows that the larger neighbourhood $\hat{\mathcal{P}}(p_0,t_0,L'+d^\#,\theta)$ does not contain surgeries as well, contradicting the maximality of $L'$. This proves $(ND2)$ holds and that the neck detection lemma can be applied to the whole neighbourhood $\hat{\mathcal{P}}(p_0,t_0,L,\theta)$.
		
		To obtain the last claim, take any $q$ such that
		\begin{align*}
		d_{g(t_0)}(q,p_0)\leq\frac{(n-1)(n-2)}{2G(p_0,t_0)}L.
		\end{align*}
		
		By the previous part of the statement, $G(q,t_0)=G(p_0,t_0)(1+O(\epsilon))>2K$. Then, Lemma \ref{Ghash}(ii) implies that $q$ has not been affected by any surgery between time $t_0-(n-1)^2d^\#/K^2$ and $t_0$. Since this holds for any $q$ in the neighbourhood, the statement is proved.
	\end{proof}
	
	In the next results we assume that our flow with surgeries satisfies certain properties. 
	\begin{enumerate}[(g1)]
		\item The surgeries take place in a region where $G$ is approximately equal to some fixed value $K^*$. More precisely, there is $K^*>2G^\#$ such that each surgery is performed at a cross section $\Sigma_{z_0}$ of a normal neck with $r(z_0)=r^*$, where $r^*=\frac{n-1}{K^*}$.
		\item The two portions of a normal neck resulting from surgery have the following properties. One portion belongs to a component of the surface which will be removed after the surgery. In the other portion, the part of the neck that has been left unchanged by the surgery will have the following structure: the first cross section which coincides with the boundary of the region changed by the surgery the mean radius satisfies $r(z)\leq(1-/11)r*$, on the last section $r(z)\geq 2r^*$ and in the sections in between $r^*\leq r(z)\leq 2r^*$.
		\item Each surgery is essential for removing a region of the surface with curvature greater than $10K^*$. That is, if we consider any of the surgeries performed at a given surgery time, we can find a component removed afterwards which contains some point with $G\geq 10K^*$, and which would not have been disconnected from the rest of the surface without the surgery.
	\end{enumerate}
	
	If the neck parameter $\epsilon_0$ is chosen small enough, then (g1) tells us that the areas modified by surgery will have $G$ between $K^*/2$ and $2K^*$ after the surgery.
	(g1) also implies that $r(z)\leq (11/10)r^*$ on the first cross section.
	Property (g3) is a natural assumption as we wish the reduce the curvature by a certain amount each time we perform surgery. Whilst (g1) and g(3) together imply that the regions with largest curvature are not the ones affected by surgery but the ones that become disconnected from the surface and removed as they have known topology.
	Lastly, property (g3) tells us that surgeries are actually performed at a certain distance away from the ends of the neck, this will be useful in the next part of the lemma.
	
	\begin{lemma}\label{neckcor}
		Consider a flow with surgeries with our usual assumptions, and (g1)-(g3). Let $L,\theta>0$ be such that $\theta\leq d^\#$ and $L\geq 20$. Then there exist $\eta_0, G_0$ such that the following property holds. Let $(p_0,t_0)$ satisfy $(ND1)$ and $(ND2)$ of the neck detection lemma, and suppose in addition that the parabolic neighbourhood $\mathcal{\hat{P}}_G(p_0,t_0,L,\theta)$ is adjacent to a surgery region. Then $(p_0,t_0)$ lies at the centre of a hypersurface neck $N$ of length at least $L-3$, which is bordered on one side by a disc $D$. The value of $G$ on $N\cup D$ at time $t_0$ is less than $5K^*$, where $K^*$ is defined above in property (g1).
	\end{lemma}
	
	\begin{proof}
		Begin by applying part $(i)$ of the neck detection lemma to find $\eta_0$ and $H_0$ such that any point $(p_0,t_0)$ satisfying $(ND1)$ and $(ND2)$ lies at the centre of an $(\epsilon,k_0-1,L,\theta)$ shrinking curvature neck. By refining the choice of $\eta_0,G_0$ we can also obtain that for all times under consideration the neck can be parametrised as a geometric neck. Now let $(p_0,t_0)$ satisfy the assumptions of the lemma for such values of $\eta_0,G_0$. Let us also set
		\begin{align*}
		r_0=\frac{(n-1)(n-2)}{2G},\; B_0=\{p\in\mathcal{M}\;|\;d_{g(t_0)}(p,p_0)\leq r_0L\}
		\end{align*}
		Our assumptions are that $B_0$ is not modified by any surgery for $t\in[t_0-\theta r_0^2,t_0]$, but that there is a point $q_0\in\partial B_0$ and a time $s_0\in[t_0-\theta r_0^2,t_0]$ such that $q_0$ lies in the closure of a region modified by surgery at time $s_0$. Our aim is to now show that the structure is not affected by the other surgeries which may occur between time $s_0$ and $t_0$.
		
		Let us denote by $D^*$ the region modified by the surgery which includes $q_0$ in its closure, and let $N^*$ be the pat of the neck left unchanged with the properties described in (g2). Let us denote by $\Sigma^*_1$ and $\Sigma_2^*$ the two components of $\partial N^*$ having mean radius less than $(11/10)r^*$ and greater than $2r^*$ respectively. By (g2), $\Sigma_1^*\partial D^*$, and so $q_0\in\Sigma_1^*$. It follows that the mean radius of $\Sigma_1^*$ is equal to $\frac{(n-1)(n-2)}{2G(q_0,s_0)}$ up to an error of order $O(\epsilon)$. Then we know that $G(p,s_0)\geq\frac{(n-1)(n-2)}{2r^*}(10/11+O(\epsilon))=K^*(10/11+O(\epsilon))$ for all $p\in B_0$, because the fully non-linear curvature $G$ is constant up to $O(\epsilon)$ on $B_0$ at any fixed time.
		
		We claim that $B_0$ must be contained in $N^*$. In fact, we know that $B_0$ has not been changed by the surgery at time $s_0$, and so it has no common points with $D^*$. If $B_0$ were not contained in $N^*$, then it would intersect the other component $\Sigma_2^*$ of $\partial N^*$. But this is impossible, since at time $s_0$ the points in $B_0$ and in $\Sigma_2^*$ have a value of $G$ respectively greater than $(10/11)K^*$. and less than $K^*/2$ up to $O(\epsilon)$.
		
		Let $z\in[z_1,z_2]$ be the parameter describing the cross sections of $N^*$, where $z=z_i$ corresponds to $\Sigma_i^*$. Then we can find a maximal interval $[a,b]\subset[z_1,z_2]$ such that the neck corresponding to $z\in[a,b]$ is centred at $p_0$ and is contained in $B_0$. Let us denote by $N_0$ this neck. Arguing as in Lemma \ref{7.9Huisken}, we can see that $N_0$ has a length at least $L-2$.
		
		Let us now denote with $N'$ the part of $N^*$ corresponding to $z\in[z_1,a]$. Then we have that $p_0$ belongs to $N_0$, which is a normal $k_0$-hypersurface neck of length at least $L-2$, and which is bordered on one side by the region $N'\cup D^*$, which is diffeomorphic to a disc. This is the statement of our theorem, except it holds at time $s_0$ rather than the final $t_0$.
		
		It remains to show that, if there are any surgeries between time $s_0$ and $t_0$, they do not affect the region $N_0\cup N'\cup D^*$. Observe that in this region $G(p,s_0)\leq 2K^*$ for any $p$ in this region. By our choice of $D^\#, G^\#$ we have $G(p,t)\leq 4K^*$ for any $p\in N_0\cup N'\cup D^*$ and $t$ between $s_0$ and either $t_0$ or the first surgery time, if it exists, that affect this region. But this shows that there cannot be any such surgery. In fact, observe that $N_0$ is contained in $B_0$, which by assumption is not changed by surgeries in $[s_0,t_0]$. The neck $N_0$ disconnects the region  $D^*\cup N'$ from the rest of the manifold. By (g3), if a surgery changes this part, it must disconnect a region contained in $N_0\cup N'\cup D^*$  here the maximum of the curvature is at least $10K^*$. This contradicts the bound on the curvature we just found, which proves that the topology of the region does not change up to time $t_0$, and that the curvature remains below the value $5K^*$ in this region. To conclude the proof, it suffices to parametrise the geometric neck $N_0$ in normal form at the final time $t_0$, using the property that $N_0\subset B_0$ which is an $(\epsilon,k_0-1)$ curvature neck at any fixed time.
	\end{proof}
	
	In the neck detection lemma we assume that the the point under consideration $\frac{\lambda_1}{G}$ is small. The next result will allow us the deal with the case where it is not and can be reduced in some sense to the former one. It is a general property of hypersurfaces and not related to geometric flows, so the proof is exactly as in \cite{huisken2009mean}.
	
	\begin{theorem}\label{Glambdabig}
		Let $F:\mathcal{M}\to\mathbb{R}^{n+1}$, with $n>1$ be a smooth connected immersed hypersurface (not necessarily closed). Suppose that there exist $c^\#,G^\#>0$ such that $|\nabla G(p)|\leq c^\# G^2(p)$ for all $p\in\mathcal{M}$ such that $G(p)\geq G^\#$. Then, for any $\eta_0>0$ we can find $\alpha_0>0$ and $\gamma_0>1$ (depending only on $c^\#$ and $\eta_0$), such that the following holds. Let $p\in\mathcal{M}$ satisfy $\lambda_1(p)>\eta_0 G(p)$ and $G(p)\geq\gamma_0G^\#$. Then either $\mathcal{M}$ is closed with $\lambda_1>\eta_0G>0$ everywhere, or there exists a point $q\in\mathcal{M}$ such that
		\begin{enumerate}[(i)]
			\item $\lambda_1(q)\leq\eta_0G(q)$
			\item $d(p,q)\leq\frac{\alpha_0}{G(p)}$
			\item $G(q')\geq G(p)/\gamma_0$ for all $q'\in\mathcal{M}$ such that $d(p,q')\leq\alpha_0/G(p)$; in particular, $G(q)\geq G(p)/\gamma_0$.
		\end{enumerate}
	\end{theorem}
	
	\begin{proof}
		Given $\alpha_0>0$, set $\gamma_0=1+c^\#\alpha_0$. For a given $p\in\mathcal{M}$, let us set $\mathcal{M}_{p,\alpha_0}=\{q\in\mathcal{M}| d(p,q)\leq\alpha_0/G(p)\}$. By \ref{6.6HSforG}, we obtain that, if $G(p)\geq\gamma_0G^\#$, then
		\begin{align*}
		G(q)\geq\frac{G(p)}{1+c^\#d(p,q)G(p)}\geq\frac{G(p)}{\gamma_0}
		\end{align*}
		for all $q\in\mathcal{M}_{p,\alpha_0}$.
		
		To prove the theorem suppose now that $p\in\mathcal{M}$ is such that $G(p)\geq\gamma_0G^\#$ and that $\lambda_1(q)>\eta_0G(q)$ for all $q\in\mathcal{M}_{p,\alpha_0}$. We claim that is $\alpha_0$ is suitably large, these properties imply that $\mathcal{M}$ coincides with $\mathcal{M}_{p,\alpha_0}$ and is therefore compact with $\lambda_1>\eta_0 G$ everywhere. 
		
		To prove this, we show that the Gauss map $\nu:\mathcal{M}_{p,\alpha_0}\to S^n$ is surjective. Take any $\omega\in S^n$, such that $\omega\neq \pm\nu(p)$. We consider the curve $\gamma$ solution of the ODE, $\dot{\gamma}=\frac{\omega^T(\gamma)}{|\omega^T(\gamma)|}$ with $\gamma(0)=p$, where for any $q\in\mathcal{M}, \omega^T(q)=\omega-\langle\omega,\nu(p)\rangle\nu(q)$ is the component of $\omega$ tangential to $\mathcal{M}$ at $q$. Since $\dot{\gamma}|=1$, the curve $\gamma$ will be parametrized by arc length The curve can be continued until $|\omega^T(\gamma)|\neq0$, i.e. $\nu(\gamma)\pm\omega$. As long as $\gamma(s)$ is contained in $\mathcal{M}_{p,\alpha_0}$, which is the case if $s\in[0,\alpha_0/G(p)]$, we can use the property $\lambda_1>\eta_0G$ to derive some estimate.
		
		Namely, if we take an orthonormal basis $e_1,\dots ,e_n$ of the tangent space to $\mathcal{M}$ at point $\gamma(s)$, we have
		\begin{align*}
		\frac{d}{ds}\langle\nu,\omega\rangle&=\sum_{i=1}^n\langle\dot{\gamma},e_i\rangle\langle\nabla_{e_i}\nu,\omega\rangle=\frac{1}{\omega^T}h_{ij}\langle\omega,e_i\rangle\langle\omega,e_j\rangle\\
		&\geq\frac{1}{\omega^T}\eta_0H|\omega^T|^2=\eta_0G\sqrt{(1-\langle\nu,\omega\rangle^2)},
		\end{align*}
		which implies $\frac{d}{ds}arcsin\langle\nu,\omega\rangle\geq\eta_0 G$.
		
		Now suppose that $\gamma(s)$ exists for $s\in[0,\alpha_0/G(p)]$. Then we have
		\begin{align*}
		\pi&>\arcsin\langle\nu(\gamma(\frac{\alpha_0}{G(p)})),\omega\rangle-\arcsin\langle\nu(p),\omega\rangle\geq\eta_0\int_0^{\alpha_o/G(p)}G(\gamma(s))ds\\
		&\geq\eta_0\int_0^{\alpha_0/G(p)}\frac{1}{G(p)^{-1}+c^\#s}=\frac{\eta_0}{c^\#}\ln(1+c^\#\alpha_0).
		\end{align*}
		
		Thus, if $\alpha_0$ is large enough to have
		\begin{center}
			$\alpha_0>\frac{1}{c^\#}(\exp(\frac{c^\#\pi}{\eta_0})-1)$
		\end{center}
		we obtain a contradiction. Therefore there exists $s^*\in(0,\alpha_0/G(p)]$ such that either $\nu(\gamma(s))\cdot\omega\to1$ or $\nu(\gamma(s))\cdot\omega\to-1$ as $s\to s^*$. Since $\arcsin\langle\nu,\omega/rangle$ is increasing, only the first possibility can occur. This shows that $\gamma(s)$ converges, as $s\to s*$, to some point $q^*\in\mathcal{M}_{p,\alpha_0}$ such that $\ni(q^*)=\omega$, as desired.
		
		It remains to consider the case when $\omega=\pm\nu(p)$, when $\omega$ trivially belongs to the image of the Gauss map. If instead we have $\omega=-\nu(p)$, it suffices to replace $p$ with another point $p'$ sufficiently close to $p$; by convexity, we have $\nu(p')\neq\nu(p)=-\omega$ and the previous argument can be applied. Thus, we have proved that the Gauss map is surjective from $\mathcal{M}_{p,\alpha_0}$ to $S^n$. Since $\lambda_1>0$ on $\mathcal{M}_{p,\alpha_0}$, the Gauss map is also a local diffeomorphism. Then, since $S^n$ is simply connected for $n>1$, it follows that the map is a global diffeomorphism.
	\end{proof}
	
	Putting all this together, we are able to provide a result about the existence of necks before a first singular time is approached.
	
	\begin{corollary}
		Let $\mathcal{M}_t$ be a smooth G-flow of closed $2$-convex hypersurfaces. Given neck parameters $\epsilon,k,L$ there exists $G^*$ (depending on $\mathcal{M}_0$) such that, if $G_{max}(t_0)\geq G^*$, then the hypersurface at time $t_0$ either contains an $(\epsilon,k,L)$-hypersurface neck or it is convex.
	\end{corollary}
	
	\begin{proof}
	  Combine Lemma (\ref{neckcor}) with Theorem (\ref{Glambdabig}). Since we assume the flow is smooth, the parabolic neighbourhood in hypothesis $(ND2)$ trivially does not contain surgeries.
	\end{proof}
	
	Therefore, unless the surface becomes convex, we are able to perform surgery before the singular time. Section $8$ of \cite{huisken2009mean} then describes the process by which after performing surgery, the maximum of the curvature will decrease making our surface less singular, it also goes on to proving the existence of necks after the first surgery, as we have to check that condition $(ND2)$ is still satisfied.
	
	Now that we've obtained all our results regarding neck detection. One can then continue as in \cite{huisken2009mean} to prove the following theorem.
	
	\begin{theorem}
		Let $ \mathcal{M}_0$ be a smooth closed two-convex hypersurface immersed in $\mathbb{R}^{n+1}$, with $n\geq 3$. Then there exist constants $G_1<G_2<G_3$ and a G-flow with surgeries starting from $\ \mathcal{M}_0$ with the following properties:
		\begin{itemize}
			\item Each surgery takes place at a time $T_i$ such that $\max G(\cdot,T_i-)=G_3$.
			\item After the surgery, all the components of the manifold (except for ones diffeomorphic to spheres of to $S^{n+1}\times S^1$ which are neglected afterwards) satisfy $\max G(\cdot,T_i+)\leq G_2$.
			\item Each surgery starts from a cross section of a normal hypersurface neck with mean radius $r(z_o)=\frac{(n-1)(n-2)}{2G}$.
			\item The flow with surgeries terminates after finitely many steps.
		\end{itemize}
		The constants $G_i$ can be any values such that $G_1\geq\omega_1$,$G_2=\omega G_1$ and $G_3=\omega_3G_2$, with $\omega_i>1$.
	\end{theorem}

\nocite{*}
\bibliographystyle{plain}
\bibliography{references}

\end{document}